\documentclass[dvips]{radhazu}
\usepackage{radhazu}
\usepackage{graphicx}
\usepackage{units}
\usepackage{multirow}
\usepackage{amssymb}
\usepackage{ulem}
\usepackage[T1]{fontenc}

\newtheorem{fact}[theorem]{Fact}

\begin{document}

\title[Orthogonal Projection of a Infinite Round Cone]{Orthogonal Projection of an Infinite Round Cone in Real Hilbert Space}

\author[M. Kosor]{Mate Kosor}
\address[Mate Kosor]{Department of Maritime Studies\\University of Zadar\\ 23 000 Zadar, Croatia}
\email[Mate Kosor]{makosor@unizd.hr}



\abstract{We fully characterize orthogonal projections of infinite right circular (round) cones in real Hilbert spaces. 
Another interpretation is that, 
given two vectors in a real Hilbert space, we establish the optimal estimate on the angle between the orthogonal projections of the two vectors.
The estimate depends on the angle between the two vectors and the position of only one of the two vectors. Our results  
 also make a contributions to Cauchy-Bunyakovsky-Schwarz type inequalities.}

\maketitle

\section{Introduction and literature overview}

Let us introduce the topic in two simple settings.
\begin{example}
\label{exa:2Dcone}Let $C_{\mathbb{R}^{2}}(v,\varphi)=\left\{ u\in\mathbb{R}^{2}\::\: \left\langle u, v\right\rangle \geq\cos\varphi\left\Vert u\right\Vert \left\Vert v\right\Vert \right\} $.
We call it a filled angle or a one-sided infinite cone in $\mathbb{R}^{2}$. Let $V$ be any line trough
the origin and $P$ any projection on $V$. Not just orthogonal, but any projection of $C(v,\varphi)$
on $V$ can be characterized as either a whole line, or a closed half-line (bounded with
the origin), or just the origin.\end{example}

From this point on, we investigate orthogonal projections only.

\begin{problem}
\label{prob:3Dcone}The one-sided right circular (round) infinite cone in
$\mathbb{R}^{3}$ with apex in the origin, half-aperture $\varphi\in\left[0,\pi\right]$,
and axis direction given by vector $v\in\mathbb{R}^{3}$ is defined
by 
\begin{equation}
C_{\mathbb{R}^{3}}(v,\varphi)\overset{{\rm def}}{=}\left\{ u\in\mathbb{R}^{3}\::\: \left\langle u, v\right\rangle \geq\cos\varphi\left\Vert u\right\Vert \left\Vert v\right\Vert \right\} .\label{eq:3Dcone_definition}
\end{equation}
Let $V$ be a two-dimensional subspace in $\mathbb{R}^{3}$. What is the orthogonal
projection of $C_{\mathbb{R}^3}(v,\varphi)$
onto $V$? If the orthogonal projection is a round infinite cone in $V$, what is its direction
and aperture?
\end{problem}
The aim of this paper is to solve a generalization of problem \ref{prob:3Dcone}
to real Hilbert spaces. Our main result is Theorem \ref{thm:cone_projection}. When applied
to problem \ref{prob:3Dcone}, Theorem \ref{thm:cone_projection} 
distinguishes among three cases, based on $\varphi$ (the angle between $u$ and $v$)
and the angle between $v$ and $V$. Interestingly, orthogonal projections of round infinite cones
in any real Hilbert space produce only these three cases already present in three dimensional Euclidean space.
\begin{enumerate}
\item In the case when $\angle\left(v,V\right)>\frac{\pi}{2}-\varphi$,
the orthogonal projection of the cone is the whole subspace $V$. 
\item When $\angle\left(v,V\right)<\frac{\pi}{2}-\varphi$, the orthogonal projection
$P\left[C_{\mathbb{R}^{3}}(v,\varphi)\right]$ is $C_{\mathbb{R}^{2}}(P\, v,\varphi_{1})$,
which is a cone in $V$ with apex in the origin, the axis direction
given by $P\,v$ and half-aperture%
\footnote{Formula for $\varphi_{1}$ in Theorem \ref{thm:cone_projection} is
in terms of $\angle\left(v,V^{\bot}\right)$.%
} 
\begin{equation}
\varphi_{1}=\arccos\sqrt{\frac{\cos^{2}\varphi-\sin^{2}\angle\left(v,V\right)}{1-\sin^{2}\angle\left(v,V\right)}}\:.\label{eq:projected_angle}
\end{equation}

\item The border case, when $\angle\left(v,V\right)=\frac{\pi}{2}-\varphi$,
further depends on $\varphi$. 

\begin{enumerate}
\item When $\varphi=0$ and $v\perp V$, then $P\left[C_{\mathbb{R}^{3}}(v,0)\right]=\left\{ \left(0,0,0\right)\right\} $. 
\item When $v\in V$ and $\varphi=\nicefrac{\pi}{2}$, then $P\left[C_{\mathbb{R}^{3}}(v,\nicefrac{\pi}{2})\right]$
is a cone in $V$ with half-aperture $\nicefrac{\pi}{2}$ and axis given by
$P v$ (a closed half-space in $V$). 
\item When $0<\varphi<\nicefrac{\pi}{2}$ and $\angle\left(v,V\right)=\frac{\pi}{2}-\varphi$,
then the projection $P\left[C_{\mathbb{R}^{3}}(v,\varphi)\right]$
is the union of the interior of the half-space in $V$ oriented by
$P\,v$ and the origin: 
\[
P\left[C_{\mathbb{R}^{3}}(v,\varphi)\right]=\left\{ \left(0,0,0\right)\right\} \cup\left\{  y\in V\,:\, \left\langle y, Pv \right\rangle >0\right\} \mbox{.}
\]
\end{enumerate}
\end{enumerate}
Cones are well known objects in Hilbert space \cite[p.~86]{Conway1990}, of
which round cones (right circular cones) are a special case. Infinite round cone is
a rotational body with a filled angle as a radial cross-section. 
Therefore, the problem of round cone projection is related to problems of
planar angle projection,
which have long been studied in three-dimensional (3D) Euclidean space 
\cite{Rickey1937,Bodnarescu1956,Vigrand-1,Vigrand-2,Vigrand-3,Goldman1989}.
\cite{Rickey1937} and \cite{Bodnarescu1956} studied
the relationship between a fixed planar angle and the orthogonal projection
of that angle to another plane, which is another planar angle. Their
method was  to set up the appropriate
coordinate system and then analytically calculate the formulas for
value and position of the projected angle. These formulas could be
applied to the rotational body in order to deduce formula (\ref{eq:projected_angle})
in 3D Euclidean space. On the other hand, the projection of an
infinite round cone in 3D could be computed by "extending to infinity" the projected area of a 
finite right circular cone \cite{PennelDeignan1989}. Note that
applying coordinatization to circular cones in \cite{CircularConeIneq2013} provided results
valid only in finite dimensions.
 
Our method is different and suits naturally to general Hilbert space setting. 
We express geometric intuition from 3D in
terms of the standard inner product calculus and real analysis. 
 In equation (\ref{eq:3Dcone_definition}),
a round cone in 3D is defined using a so-called reverse Cauchy-Bunyakovsky-Schwarz
(CBS) type inequality, and we will use the same inequality to define a
round cone in general Hilbert space (Section \ref{sec:main_results}). 
In Section \ref{sec:Few-Applications}, we show few applications, among which Example \ref{inf-exmpl}
is in (infinite dimensionsional) Lebesgue space.  

Fifty years ago, paper \cite[p.~89]{DiazMetcalf1966} mentioned connection between cones
 and reverse\footnote{In \cite{DiazMetcalf1966} the inequality is called 
"complementary" to CBS or "running the other way".} triangle inequality 
in Hilbert spaces. Thus, it is not surprising that 
results on round infinite cone projections are directly
related to new CBS-type inequalities (Section \ref{sec:inequalities}).  

Known reverse CBS inequalities \cite{Dragomir2003} provide ways
to estimate $\cos\varphi$ in (\ref{eq:3Dcone_definition}) from below,
based on some knowledge about the projections of $u$ and $v$. For example, the P\'{o}lya-Szeg\"{o} inequality in $\mathbb{R}^n$ \cite{POLYA1970} estimates $\cos\varphi$ in 
(\ref{eq:3Dcone_definition}) based on lower and upper bounds 
of coordinates $m_{u}\leq u_{i}\leq M_{u}$ and $m_{v}\leq v_{i}\leq M_{v}$. 
Cassels' inequality \cite[page 330]{WATSON1955} and its refinement by Andrica and 
Badea \cite{ANDRICA1988} provide a bound on $\cos\varphi$ in (\ref{eq:3Dcone_definition})
based on the bounds of the ratio $m\leq\nicefrac{u_{i}}{v_{i}}\leq M$. On
the other hand, in this paper we estimate the angle (\ref{eq:projected_angle})
between orthogonal projections $Pu$ and $Pv$ based on the value of $\cos\varphi$
in (\ref{eq:3Dcone_definition}). Our result on a reverse CBS inequality, 
Theorem \ref{thm:inequality}, by contraposition  gives
a sufficient condition for an estimate that is more strict then the classical CBS:
$\left\langle u,v\right\rangle \leq\alpha\|u\|\|v\|$ (see 
Example \ref{exa:CBS_enhanced}).

\section{Assumptions, Notation and the Main result\label{sec:main_results}}

Throughout the remainder of the paper we make two assumptions.
\begin{enumerate}
\item $H$ is a real Hilbert space, $\left\Vert x\right\Vert =\sqrt{\left\langle x,x\right\rangle }$
denotes vector norm and $O$ denotes zero vector, 
\item $V$ is a closed subspace of $H$, and $V^{\bot}$ denotes its orthogonal
complement.
\end{enumerate}
From classical Hilbert space theory we know that there exists the unique
orthogonal projection onto $V$, which we denote by $P\,:\, H\to V$.
Given any set $\Omega\subseteq H$, its orthogonal projection onto
$V$ is denoted by $P\left[\Omega\right]$. We define angles between
vectors $u$ and $v$, and between vector $u$ and subspace $S$ with%
\footnote{Angle definition in (\ref{eq:angle_vectors}) allows the trivial cases $V=H$,
$V=\left\{ O\right\}$, and $v=O$ to be naturally included in Theorem \ref{thm:cone_projection},
without any special considerations. We use sign $\overset{{\rm def}}{=}$ throughout the paper 
in order to indicate that the relation is actually a definition.%
} 
\begin{eqnarray}
\angle(u,v) & \overset{{\rm def}}{=} & \sup\left\{ \varphi\in\left[0,\pi\right]\,:\,\left\langle u,v\right\rangle \leq\cos\varphi\left\Vert u\right\Vert \left\Vert v\right\Vert \right\} ,\label{eq:angle_vectors}\\
& = & \begin{cases}
\arccos\frac{\left\langle u,v\right\rangle }{\left\Vert u\right\Vert \left\Vert v\right\Vert }, & \mbox{if }u\neq O\mbox{ and }v\neq O,\\
\pi, & \mbox{if }u=O\mbox{ or }v=O.
\end{cases} \\
\angle(u,S) & \overset{{\rm def}}{=} & \inf\left\{ \angle(u,v)\,:\, v\in S\right\} \label{eq:angle_vector_subspace}
\end{eqnarray}
Thus $\angle(u,\left\{ O\right\} )=\pi=\angle(O,V)$.
Note that $\angle\left(v,V^{\bot}\right)+\angle\left(v,V\right)=\frac{\pi}{2}$,
except when $v=0$, or $V=H$, or $V=\left\{ O\right\} $. 
\begin{definition}
\label{def:Directed-cone}The infinite one-sided solid right circular (round) cone in $H$
with apex $a\in H$, axis direction given by $v\in H$, and half-aperture
$\varphi\in[0,\pi]$ is defined by
\[
C_{H}\left(a,v,\varphi\right)\overset{{\rm def}}{=}\left\{ u\in H\,:\,\left\langle u-a,v\right\rangle \geq\cos\varphi\,\left\Vert u-a\right\Vert \left\Vert v\right\Vert \right\} \mbox{.}
\]
The infinite one-sided solid right circular (round) cone with the apex included but with the rest
of the boundary excluded is defined by 
\[
C_{H}^{\circ}\left(a,v,\varphi\right)\overset{{\rm def}}{=}\left\{ a\right\} \cup\left\{ u\in H\,:\,\left\langle u-a,v\right\rangle >\cos\varphi\,\left\Vert u-a\right\Vert \left\Vert v\right\Vert \right\} \mbox{.}
\]

\end{definition}
When the apex is $O$, the notation is abbreviated: $C_{H}\left(v,\varphi\right)\overset{{\rm def}}{=}C_{H}\left(O,v,\varphi\right)$
and $C_{H}^{\circ}\left(v,\varphi\right)\overset{{\rm def}}{=}C_{H}^{\circ}\left(O,v,\varphi\right)$.
By definition, we have $C_{H}\left(a,O,\varphi\right)=H$ and $C_{H}^{\circ}\left(a,O,\varphi\right)=\left\{ a\right\} $.
By CBS inequality, we have $C_{H}\left(a,v,\pi\right)=H$ and $C_{H}^{\circ}(a,v,0)=\{a\}$.
Dilation (Minkowski addition) is denoted by $X+Y=\{x+y\,:\, x\in X\mbox{ and }y\in Y$\}.
Thus, $C_{H}\left(a,v,\varphi\right)=C_{H}\left(v,\varphi\right)+\left\{ a\right\} $. 

Given $a\in H$, $v\in H$ and $\varphi\in[0,\pi]$, the aim
of this paper is to determine the projection $P\left[C_{H}\left(a,v,\varphi\right)\right]$.
The main result is the following theorem.
\begin{theorem}
\label{thm:cone_projection}Let $C_{H}\left(a,v,\varphi\right)$ be
an infinite one-sided solid cone in $H$ and 
\begin{equation}
\varphi_{1}=\begin{cases}
\arccos\sqrt{{\displaystyle \frac{\cos^{2}\varphi-\cos^{2}\angle\left(v,V^{\bot}\right)}{1-\cos^{2}\angle\left(v,V^{\bot}\right)}}}, & \angle\left(v,V^{\bot}\right)\in\left\langle 0,\frac{\pi}{2}\right]\\
\varphi, & \mbox{else}.
\end{cases}\label{eq:phi1}
\end{equation}
Given the values of $\varphi$ and $\angle\left(v,V^{\bot}\right)$
in the first two columns of the following table, the orthogonal projection $P$ of $C_{H}\left(a,v,\varphi\right)$
onto a closed subspace $V$ can be determined from the third column of the same table.

\begin{center}
\begin{tabular}{|c|c||c|}
\hline 
$\varphi$ & $\angle\left(v,V^{\bot}\right)$ & ${\displaystyle \vphantom{\dfrac{X}{X}}P\left[C_{H}\left(a,v,\varphi\right)\right]}$\tabularnewline
\hline 
\hline 
$\varphi=0$ & $\angle\left(v,V^{\bot}\right)=0$ & $\vphantom{\dfrac{X}{X}}\{P\, a\}$\tabularnewline
\hline 
\multirow{2}{*}{$\varphi=\angle\left(v,V^{\bot}\right)$} & $\angle\left(v,V^{\bot}\right)\in\left\langle 0,\nicefrac{\pi}{2}\right\rangle $ & $\vphantom{\dfrac{X}{X}}C_{V}^{\circ}(P\, a,\, P\, v,\,\nicefrac{\pi}{2})$\tabularnewline
\cline{2-3} 
 & $\vphantom{\dfrac{X}{X}}\angle\left(v,V^{\bot}\right)\geq\nicefrac{\pi}{2}$ & \multirow{2}{*}{${\displaystyle \vphantom{\dfrac{X}{X}}C_{V}(P\, a,\, P\, v,\,\varphi_{1})}$}\tabularnewline
\cline{1-2} 
$\varphi<\angle\left(v,V^{\bot}\right)$ & $\vphantom{\dfrac{X}{X}}\angle\left(v,V^{\bot}\right)>0$ & \tabularnewline
\hline 
\multicolumn{2}{|c||}{$\varphi>\angle\left(v,V^{\bot}\right)$} & $\vphantom{\dfrac{X}{X}}V$\tabularnewline
\hline 
\end{tabular}
\par\end{center}

\end{theorem}
\begin{proof}
Suppose that the apex is $a=O$. We prove the theorem on a case by case basis.
Technical work is deferred to Section \ref{sec:inequalities}, which
deals with reverse CBS inequalities that underlay the definition of
the cone. 

\uline{Case $v=O$.} This case is trivial as $C_{H}\left(O,\varphi\right)=H$,
$P\left[C_{H}\left(O,\varphi\right)\right]=V=C_{V}\left(O,\varphi_{1}\right)$,
$\angle\left(O,V^{\bot}\right)=\pi$, and $\varphi_{1}=\varphi$.

\uline{Case $\dim V=0$.} Projection collapses
everything to $V=\{O\}$, $\varphi_{1}=\varphi$, and $\angle\left(v,V^{\bot}\right)=0$.
We have either $\varphi=\angle\left(v,V^{\bot}\right)=0$, or $\varphi>\angle\left(v,V^{\bot}\right)$
and both corresponding rows in the table of the theorem provide
the valid answer.

\uline{Case $V=H$.} Let $C_{H}\left(v,\varphi\right)$ be a cone,
such that $v\neq O$. We have $P\, v=v$, $P\left[C_{H}\left(v,\varphi\right)\right]=C_{V}\left(v,\varphi\right)$,
$\angle\left(v,V\right)=0$, $\angle\left(v,V^{\bot}\right)=\pi$
and $\varphi\leq\angle\left(v,V^{\bot}\right)$. From formula (\ref{eq:phi1}),
$\varphi_{1}=\varphi$. Thus, $P\left[C_{H}\left(v,\varphi\right)\right]=C_{H}(P\, v,\,\varphi)$,
which is the conclusion of the theorem.

\uline{Case $\varphi=0$.} Note that $\cos\varphi=1$, $C_{H}\left(v,0\right)=\left\{ t\, v\,:\, t\geq0\right\} $
and $P\left[C_{H}\left(v,0\right)\right]=\left\{ t\, P\, v\,:\, t\geq0\right\} =C_{V}\left(P\, v,0\right)$
unless $P\ v=O$. Note also that formula (\ref{eq:phi1})
produces $\varphi_{1}=0$ when $\varphi=0$. In the special case when
$P\, v=O$, then $\angle\left(v,V^{\bot}\right)=0$ and $P\left[C_{H}\left(v,0\right)\right]=\left\{ O\right\} $.

\uline{Case $\varphi<\angle\left(v,V^{\bot}\right)$.} As $\varphi\geq0$,
we must have $\angle\left(v,V^{\bot}\right)>0$. Therefore $P\, v\neq O$.
The cases $\dim V=0$, and $V=H$ have already been solved.  
In this case the main part of the proof is Theorem
\ref{thm:inequality}. The first part of Theorem \ref{thm:inequality}
states that: $u\in C_{H}\left(v,\varphi\right)$ implies $P\, u\in C_{V}\left(P\, v,\,\varphi_{1}\right)$,
i.e.~$P\left[C_{H}\left(v,\varphi\right)\right]\subseteq C_{V}\left(P\, v,\,\varphi_{1}\right)$.
In the subcase $\dim V\geq2$, the second part of Theorem \ref{thm:inequality}
establishes existence of $u\in C_{H}\left(v,\varphi\right)$ such
that $P u\neq O$, and $\left\langle P\, u,P\, v\right\rangle =\cos\varphi_{1}\,\left\Vert P\, u\right\Vert \left\Vert P\, v\right\Vert $.
By Lemma \ref{lem:extreme-bound-to-projected-cone}, we get $C_{V}\left(P\, v,\,\varphi_{1}\right)\subseteq P\left[C_{H}\left(v,\varphi\right)\right]$,
and the subcase $\dim V\geq2$ is solved.

In the subcase $\dim V=1$, there are just 2 unit vectors in $V$,
which we denote with $1_{V}$ and $-1_{V}$. Also, there are just 4
different ``cones'' with the apex $O$ in $V$: $\left\{ 0\right\} $,
$V$, $C_{V}\left(1_{V},0\right)$ and $C_{V}\left(-1_{V},0\right)$
(cf.\ Example \ref{exa:2Dcone}). Therefore $C_{V}(P\, v,\nicefrac{\pi}{2})=C_{V}(P\, v,\varphi_{1})=C_{V}\left(P\, v,0\right)$,
as $\varphi_{1}\in\left[0,\nicefrac{\pi}{2}\right]$. Without the
loss of generality, we can assume that $1_{V}$ and $P\, v$ are pointing in the same direction
(we are free to swap the names between $1_{V}$ and $-1_{V}$). Thus,
we get $C_{V}\left(1_{V},0\right)=\left\{ t\, Pv\,:\, t\in\mathbb{R}\right\} \subseteq P\left[C_{H}\left(v,\varphi\right)\right]$
and $C_{V}\left(P\, v,0\right)=C_{V}\left(1_{V},0\right)$. Together,
we have $C_{V}(P\, v,\varphi_{1})\subseteq P\left[C_{H}\left(v,\varphi\right)\right]$.
As noted before, from Theorem \ref{thm:inequality} we get $P\left[C_{H}\left(v,\varphi\right)\right]\subseteq C_{V}\left(P\, v,\,\varphi_{1}\right)$. 
Therefore, $P\left[C_{H}\left(v,\varphi\right)\right]=C_{V}\left(P\, v,\,\varphi_{1}\right)$.

\uline{Case $\varphi>\angle\left(v,V^{\bot}\right)$.} This implies
$V\neq H$, because of the angle definition in (\ref{eq:angle_vectors}).
The main part of the proof in this case is moved to Proposition \ref{prop:phi_geq_psi}
and Lemma \ref{lem:extreme-bound-to-projected-cone}. Proposition
\ref{prop:phi_geq_psi} shows that there is $u\in C_{H}\left(v,\varphi\right)$
such that $P u \neq O$, and $\left\langle P\, u,P\, v\right\rangle =\cos\pi\,\left\Vert P\, u\right\Vert \left\Vert P\, v\right\Vert $.
From there, Lemma \ref{lem:extreme-bound-to-projected-cone} concludes
that $C_{V}\left(P\, v,\pi\right)\subseteq P\left[C_{H}\left(v,\varphi\right)\right]$.
By the definition of cone and by the CBS inequality we know that $C_{V}\left(Pv,\pi\right)=V$.
Thus we get $P\left[C_{H}\left(v,\varphi\right)\right]=V$, which
is just what the theorem states in this case.

\uline{Case $\varphi=\angle\left(v,V^{\bot}\right)$.} Because of the other cases that have 
been discussed already, we can safely assume that
$\varphi=\angle\left(v,V^{\bot}\right)\in\left\langle 0,\nicefrac{\pi}{2}\right]$
and $\dim V\geq1$.
Note that $Pv\neq O$ and $\varphi_{1}=\nicefrac{\pi}{2}$.

First, suppose $\dim V=1$. As we noted earlier, there are just four
cones to distinguish in $V$. Note that $C_{V}^{\circ}(P\, v,\nicefrac{\pi}{2})=C_{V}(P\, v,\nicefrac{\pi}{2})=C_{V}\left(P\, v,0\right)$.
From $Pv\neq O$ we get $C_{V}\left(P\, v,0\right)\subseteq P\left[C_{H}\left(v,\varphi\right)\right]$.
From Lemma \ref{lemm:phi_equal_psi-helper} we get $P\left[C_{H}\left(v,\varphi\right)\right]\subseteq C_{V}(P\, v,\nicefrac{\pi}{2})$.
Thus, we conclude $P\left[C_{H}\left(v,\varphi\right)\right]=C_{V}(P\, v,\nicefrac{\pi}{2})=C_{V}^{\circ}(P\, v,\nicefrac{\pi}{2})$.

Next, suppose $\dim V\geq2$ and $V\neq H$. We discuss two subcases:
either $\varphi=\angle\left(v,V^{\bot}\right)\in\left\langle 0,\nicefrac{\pi}{2}\right\rangle $
or $\varphi=\angle\left(v,V^{\bot}\right)=\nicefrac{\pi}{2}$. Suppose
$\varphi=\angle\left(v,V^{\bot}\right)\in\left\langle 0,\nicefrac{\pi}{2}\right\rangle $.
By Proposition \ref{prop:phi_equal_psi-first_result} and Lemma \ref{lem:extreme-bound-to-projected-cone}:
$C_{V}\left(Pv,\nicefrac{\pi}{2}-\varepsilon\right)\subseteq P\left[C_{H}\left(v,\varphi\right)\right]$,
for each $\varepsilon\in\left\langle 0,\nicefrac{\pi}{2}\right]$.
Furthermore, Proposition \ref{prop:phi_equal_psi-second_result} yields
$C_{V}\left(-Pv,\nicefrac{\pi}{2}\right)\cap P\left[C_{H}\left(v,\varphi\right)\right]=\left\{ O\right\} $,
and therefore $P\left[C_{H}\left(v,\varphi\right)\right]=C_{V}^{\circ}\left(Pv,\nicefrac{\pi}{2}\right)$.

Finally, we prove the subcase $\varphi=\angle\left(v,V^{\bot}\right)=\nicefrac{\pi}{2}$
and $\dim V\geq2$. Then, $Pv=v$ and $C_{V}\left(v,\nicefrac{\pi}{2}\right)=C_{H}\left(v,\nicefrac{\pi}{2}\right)\cap V\subseteq P\left[C_{H}\left(v,\nicefrac{\pi}{2}\right)\right]$.
Lemma \ref{lemm:phi_equal_psi-helper} shows that $\left\langle u,v\right\rangle \geq0$
implies $\left\langle Pu,Pv\right\rangle \geq0$. Thus $P\left[C_{H}\left(v,\nicefrac{\pi}{2}\right)\right]\subseteq C_{V}\left(v,\nicefrac{\pi}{2}\right)$.
From $\varphi=\varphi_{1}=\nicefrac{\pi}{2}$ we conclude $P\left[C_{H}\left(v,\varphi\right)\right]=C_{V}\left(Pv,\varphi_{1}\right)$.

We have proved the theorem for apex $a=O$. The general case follows
from the properties of dilation 
$$
P\left[C_{H}\left(a,v,\varphi\right)\right]=P\left[C_{H}\left(v,\varphi\right)+\left\{ a\right\} \right]=P\left[C_{H}\left(v,\varphi\right)\right]+\left\{ P\, a\right\} \mbox{.}
$$
\end{proof}

\section{A Few Applications\label{sec:Few-Applications}}
\begin{remark}
Let $d$ be a vector orthogonal to $V$. Then $\Pi\overset{{\rm def}}{=}V+\{d\}$
is an affine subspace in $H$. Orthogonal projection of $u$ onto
$\Pi$ is defined by $P_{\Pi}u=P\, u + d$. Orthogonal projection
of cone $C_{H}\left(a,v,\varphi\right)$ onto $\Pi$ yields $P_{\Pi}\left[C_{H}\left(a,v,\varphi\right)\right]=P\left[C_{H}\left(a,v,\varphi\right)\right]+\left\{ d\right\} \subseteq\Pi$,
where $P\left[C_{H}\left(a,v,\varphi\right)\right]$ is given in Theorem
\ref{thm:cone_projection}.\end{remark}

\begin{remark} 
We can use Theorem \ref{thm:cone_projection} to describe projections of solid cones with the apex included but
without the rest of the boundary. It is easy to see that
\begin{equation}
P\left[ C_{H}^{\circ}\left(a,v,\varphi\right) \right] = P\left[\right. \bigcup_{\varepsilon>0} C_{H}\left(a,v,\varphi-\varepsilon \right) \left.\right] 
= \bigcup_{\varepsilon>0} P\left[ C_{H}\left(a,v,\varphi-\varepsilon \right) \right]\mbox{.}
\end{equation}
Therefore, if we extend the formula for $\varphi_1$ in \eqref{eq:phi1} by setting $\varphi_1=0$ when $\varphi=\angle\left(v,V^{\bot}\right)=0$ we get
\begin{equation}
P\left[ C_{H}^{\circ}\left(a,v,\varphi \right) \right] =
\begin{cases}
C_{V}^{\circ}\left(Pa,Pv,\varphi_1 \right), & \varphi\leq\angle\left(v,V^{\bot}\right),\\
V, & \varphi > \angle\left(v,V^{\bot}\right)\mbox{.}
\end{cases}\label{cones-with-apex-no-boundary}
\end{equation}
Furthermore, a relation equivalent to \eqref{cones-with-apex-no-boundary} 
is also valid for "open" cones. 
\end{remark}

\begin{example}
\label{ex:search_for_a_cone} Let $v\in H$ such that $Pv\neq O$.
Which is the widest half aperture $\varphi$ of an infinite solid
cone with the apex $a\in H$, axis and direction given by $v$ such that
the half aperture of a projected cone is at most $\varphi_{1}<\nicefrac{\pi}{2}$? 

Solving formula (\ref{eq:projected_angle}) for $\varphi$ yields
\begin{equation}
\varphi=\arccos\sqrt{\cos^{2}\angle\left(v,V^{\bot}\right)+\cos^{2}\varphi_{1}-\cos^{2}\angle\left(v,V^{\bot}\right)\,\cos^{2}\varphi_{1}}\mbox{.}\label{eq:aperture_relations-1}
\end{equation}
Theorem \ref{thm:cone_projection} establishes $P\left[C_{H}\left(a,v,\varphi\right)\right]=C_{V}\left(Pa,\, Pv,\,\varphi_{1}\right)$
and any larger $\varphi$ would yield aperture of projected cone larger
then $\varphi_{1}$. 
\end{example}

\begin{fact}
\label{fact:widest_angle_in_orthant}The widest half aperture of an
one-sided infinite cone that fits inside an orthant (hyperoctant) of $\mathbb{R}^{n}$
is $\varphi=\arccos\sqrt{\frac{n-1}{n}}$.
As dimension $n\to\infty$, the aperture $\varphi\to 0$.
\end{fact}
\begin{proof}
All the projections onto coordinate 2D planes $V_{ij}$ of such a infinite one-sided
cone need to fit into a quadrant, which is a directed cone with half aperture
$\varphi_{1}=\nicefrac{\pi}{4}$ around directed axis $P\, v=(1,1)$.
Therefore, the directed cone with the widest aperture needs to
have the axis $v=(1,1,\ldots,1)$. By formula (\ref{eq:definition-Psi})
$\angle\left(v,V_{ij}\right)=\arccos\frac{\sqrt{2}}{\sqrt{n}}$ and
$\cos^{2}\angle\left(v,V_{ij}^{\bot}\right)=\frac{n-2}{n}$. Formula
(\ref{eq:projected_angle}) yields $\varphi=\arccos\sqrt{\frac{n-1}{n}}$
and Theorem \ref{thm:cone_projection} establishes the fact.\end{proof}

\begin{example}\label{inf-exmpl}
Let $\alpha\in\left(0,1\right)$ and  $H=L^{2}\left(0,1\right)$ Lebesgue space.
What is the smallest $t>0$ such that for all $u\in H$,
\begin{equation}
\underbrace{\;\int_{0}^{1}u\geq\alpha\sqrt{\int_{0}^{1}u^{2}}\;}_{\star} \quad\implies\quad \int_{0}^{t}u\geq0\quad ?
\label{question}
\end{equation}
 
Let $v={\bf 1}_{(0,1)}\in H$, where we denote with ${\bf 1}_{X}$ the characteristic function of $X$. 
Then $u\in C_H(v,\arccos\alpha)$ if and only if $u$
satisfies $\star$ in \eqref{question}.
Let $V_t=\{ f\in L^{2}\left(0,1\right) : f(x)=0 \mbox{ for almost all } x\in\left(t,1\right) \}$.
$V_t$ is a closed subspace of $H$ and isometrically isomorphic with $L^{2}\left(0,t\right)$. 
Let $P_{V_t}$ be the orthogonal projection onto $V_t$. Then $P_{V_t}(v)={\bf 1}_{(0,t)}$.
We calculate the
angles $\angle(v,V_t)=\angle(v,{\bf 1}_{(0,t)})=\arccos\sqrt{t}$ and 
$\angle(v,V_{t}^{\bot})=\angle(v,{\bf 1}_{(t,1)})=\arccos\sqrt{1-t}$.

From Theorem \ref{thm:cone_projection}, if $\arccos\alpha>\arccos\sqrt{1-t}$, then $P_{V_t}[C_H(v,\arccos\alpha)]=V_t$.
In other words, when $t<1-\alpha^{2}$, then for appropriate
$u$ that satisfies $\star$ in \eqref{question} we can get $\int_{0}^{t}u<0$. On the other hand, 
for $\arccos\alpha\leq\arccos\sqrt{1-t}$, $P_{V_t}[C_H(v,\arccos\alpha)]$ is either 
$C_{V_t}({\bf 1}_{(0,t)},\varphi_1)$ or $C_{V_t}^{\circ}({\bf 1}_{(0,t)},\nicefrac{\pi}{2})$,
where $\varphi_1<\nicefrac{\pi}{2}$ can be calculated from \eqref{eq:phi1} with $\varphi=\arccos\alpha$.
In any case, if $t\geq1-\alpha^{2}$,
then for any $u$ that satisfies $\star$ in \eqref{question} we have $\int_{0}^{t}u\geq\sqrt{t-1+\alpha^{2}}\sqrt{\int_{0}^{t}u^{2}}$.
Therefore we have shown that $t=1-\alpha^{2}$ is the smallest number that satisfies \eqref{question}. 
\end{example}

\section{Reverse CBS Inequalities \label{sec:inequalities}}

In this section we provide technical results used in the proof of Theorem \ref{thm:cone_projection}
on cone projections. Cones have been defined in terms of reverse CBS
inequality. Given a reverse CBS inequality in Hilbert space $H$,
we need to establish an ``optimal'' reverse CBS inequality for orthogonal
projections onto a closed subspace $V$. This is the converse from
reverse CBS inequalities in survey \cite[section 5.]{Dragomir2003}.
The contraposition of our results provides sufficient conditions
for an estimate that is more strict then the classical CBS inequality: 
$\left\langle u,v\right\rangle <\alpha\left\Vert u\right\Vert \left\Vert v\right\Vert $.
This is described in Example \ref{exa:CBS_enhanced}.
\begin{remark}
\label{rem:decomposition}Throughout this section we use subscripts
to denote $u_{1}\overset{{\rm def}}{=}P\: u$ and $u_{2}\overset{{\rm def}}{=}u-P\, u$,
for any $u\in H$. For the unique decomposition of $u\in H$ as the
sum of two orthogonal vectors, one from $V$ and other from $V^{\bot}$
we write $u=u_{1}\dotplus u_{2}$. For $u\neq O$ we use formulas
\begin{equation}
\angle\left(u,V^{\bot}\right)=\arccos\frac{\left\Vert u_{2}\right\Vert }{\left\Vert u\right\Vert }=\begin{cases}
\arctan\frac{\left\Vert u_{1}\right\Vert }{\left\Vert u_{2}\right\Vert }, & u_{2}\neq O,\\
\frac{\pi}{2}, & u_{2}=O,
\end{cases}\label{eq:definition-Psi}
\end{equation}
\end{remark}

\begin{example}
\label{exa:CBS_enhanced}Suppose we have set $v\in H$, $V$ closed
subspace of $H$, and $\alpha\in\left(0,1\right)$. We are looking for a sufficient
condition on the component $u_{1}=P\, u$, that can establish an inequality
stronger then the CBS inequality: $\left\langle u,v\right\rangle <\alpha\left\Vert u\right\Vert \left\Vert v\right\Vert$.

We use contraposition of Theorem \ref{thm:inequality}. Condition
$\alpha=\cos\varphi>\cos\angle\left(v,V^{\bot}\right)=\frac{\|v_{2}\|}{\|v\|}$
gives $\|v_{2}\|<\frac{\alpha}{\sqrt{1-\alpha^{2}}}\|v_{1}\|$. If
that condition is met, then $\left\langle u_{1},v_{1}\right\rangle <\cos\varphi_{1}\left\Vert u_{1}\right\Vert \left\Vert v_{1}\right\Vert $
is the sufficient condition for $\left\langle u,v\right\rangle <\alpha\left\Vert u\right\Vert \left\Vert v\right\Vert $,
where $\varphi_{1}$ is the same as in (\ref{eq:phi1}). Therefore,
\begin{multline}
\left(\|v_{2}\|<\frac{\alpha\|v_{1}\|}{\sqrt{1-\alpha^{2}}}\mbox{ and }
\left\langle u_{1},v_{1}\right\rangle <\sqrt{\frac{\alpha^{2}\|v\|^2-\|v_{2}\|^2}{\|v\|^2-\|v_{2}\|^2}}\left\Vert u_{1}\right\Vert \left\Vert v_{1}\right\Vert \right)\\
\implies\qquad\left\langle u,v\right\rangle <\alpha\left\Vert u\right\Vert \left\Vert v\right\Vert .
\end{multline}

\end{example}
The following result is a cornerstone in the proof of Theorem \ref{thm:cone_projection}.
\begin{theorem}
\label{thm:inequality}Let $v\in H$, $V$ closed subspace of $H$,
$\varphi$ such that $0\leq\varphi<\angle\left(v,V^{\bot}\right)$
and $\varphi_{1}$ as in (\ref{eq:phi1}). Let $\alpha=\cos\varphi$ and $\alpha_1=\cos\varphi_1$. Then for arbitrary $u\in H$,
\begin{equation}
\left\langle u,v\right\rangle \geq\alpha\,\left\Vert u\right\Vert \left\Vert v\right\Vert \qquad\Longrightarrow\qquad\left\langle P\: u,\, P\, v\right\rangle \geq\alpha_{1}\,\left\Vert Pu\right\Vert \left\Vert Pv\right\Vert \mbox{.}\label{eq:implikacija_projekcija}
\end{equation}
Moreover, when $\dim V\geq2$ and $v\neq O$ then our $\alpha_{1}$
is the largest possible in (\ref{eq:implikacija_projekcija}). In other words, when
$\dim V\geq2$, then there exists $u\in H$ such that $P u\neq O$, $\left\langle u,v\right\rangle \geq\alpha\,\left\Vert u\right\Vert \left\Vert v\right\Vert $
and $\left\langle P\: u,\, P\, v\right\rangle =\alpha_{1}\,\left\Vert Pu\right\Vert \left\Vert Pv\right\Vert $.\end{theorem}
\begin{proof}
We use notation $v=v_{1}\dotplus v_{2}$ and $u=u_{1}\dotplus u_{2}$
as in Remark \ref{rem:decomposition}. If $v_{1}=O$ or $V=H$, (\ref{eq:implikacija_projekcija})
is trivial. Thus, without the loss of generality we can suppose $v_{1}\neq O$
and $0\leq\varphi<\angle\left(v,V^{\bot}\right)\leq\nicefrac{\pi}{2}$.
Then%
\footnote{Formula numbers above and under (in)equality sign establish a cross
reference that can help to understand the relationship. This notation
is used throughout the article.%
}, 
\begin{flalign}
0 & \leq\frac{\left\Vert v_{2}\right\Vert }{\left\Vert v\right\Vert }\overset{\eqref{eq:definition-Psi}}{=}\cos\angle\left(v,V^{\bot}\right)<\cos\varphi\leq1\mbox{.}\label{eq:condition_on_angle_and_v}\\
\frac{\left\Vert v_{2}\right\Vert }{\left\Vert v_{1}\right\Vert } & =\frac{\cos\angle\left(v,V^{\bot}\right)}{\sqrt{1-\cos^{2}\angle\left(v,V^{\bot}\right)}}\label{eq:useful_identity}
\end{flalign}
When $u_{1}=O$ then (\ref{eq:implikacija_projekcija}) is trivial.
The main part of the proof investigates 
\begin{multline}
\min\quad\frac{\left\langle u_{1},v_{1}\right\rangle }{\left\Vert u_{1}\right\Vert \left\Vert v_{1}\right\Vert }=\min\quad\frac{\cos\theta\,\left\Vert u\right\Vert \left\Vert v\right\Vert -\left\langle u_{2},v_{2}\right\rangle }{\left\Vert u_{1}\right\Vert \left\Vert v_{1}\right\Vert }=\\
=\min\qquad\underbrace{\cos\theta\,\sqrt{1+\frac{\left\Vert u_{2}\right\Vert ^{2}}{\left\Vert u_{1}\right\Vert ^{2}}}\sqrt{1+\frac{\left\Vert v_{2}\right\Vert ^{2}}{\left\Vert v_{1}\right\Vert ^{2}}}-\left\langle \frac{u_{2}}{\left\Vert u_{1}\right\Vert },\frac{v_{2}}{\left\Vert v_{1}\right\Vert }\right\rangle }_{(\blacksquare)}\label{eq:minimum1}
\end{multline}
under the conditions that $u_{1}\neq O$ and $\left\langle u,v\right\rangle =\cos\theta\left\Vert u\right\Vert \left\Vert v\right\Vert \geq\cos\varphi\,\left\Vert u\right\Vert \left\Vert v\right\Vert $,
for some $\theta\in[0,\varphi]$ that depends on $u$ and $v$. Under
these conditions 
\[
(\blacksquare)\geq\cos\varphi\,\sqrt{1+\frac{\left\Vert u_{2}\right\Vert ^{2}}{\left\Vert u_{1}\right\Vert ^{2}}}\sqrt{1+\frac{\left\Vert v_{2}\right\Vert ^{2}}{\left\Vert v_{1}\right\Vert ^{2}}}-\frac{\left\Vert u_{2}\right\Vert \left\Vert v_{2}\right\Vert }{\left\Vert u_{1}\right\Vert \left\Vert v_{1}\right\Vert }=f\left(\frac{\left\Vert u_{2}\right\Vert }{\left\Vert u_{1}\right\Vert },\frac{\left\Vert v_{2}\right\Vert }{\left\Vert v_{1}\right\Vert }\right)
\]
where ${\displaystyle f(a,b)=\cos\varphi\,\sqrt{1+a^{2}}\sqrt{1+b^{2}}-ab}$.

As $\nicefrac{\left\Vert u_{2}\right\Vert }{\left\Vert u_{1}\right\Vert }\geq0$
and $v$ has been fixed from the start, together with condition (\ref{eq:condition_on_angle_and_v}),
it is sufficient to examine the function $a\overset{g}{\longmapsto}f(a,b)$
for all $a\geq0$ and a fixed $b$, taking into account that $\cos\varphi>\nicefrac{b}{\sqrt{1+b^{2}}}.$
The continuity of $g$, the first, and the second derivative of $g$, together show that $g$ is
convex with the only minimizer $a=\nicefrac{b}{\sqrt{\cos^{2}\varphi\left(1+b^{2}\right)-b^{2}}}$
and the minimum $\sqrt{\cos^{2}\varphi\,\left(1+b^{2}\right)-b^{2}}$.
Therefore 
\[
(\blacksquare)\geq\sqrt{\cos^{2}\varphi\,\left(1+\frac{\left\Vert v_{2}\right\Vert ^{2}}{\left\Vert v_{1}\right\Vert ^{2}}\right)-\frac{\left\Vert v_{2}\right\Vert ^{2}}{\left\Vert v_{1}\right\Vert ^{2}}}\overset{(\ref{eq:useful_identity})}{=}\sqrt{\frac{\cos^{2}\varphi-\cos^{2}\angle\left(v,V^{\bot}\right)}{1-\cos^{2}\angle\left(v,V^{\bot}\right)}}>0\mbox{.}
\]
Thus ${\displaystyle \left\langle u_{1},v_{1}\right\rangle \geq\sqrt{\frac{\cos^{2}\varphi-\cos^{2}\angle\left(v,V^{\bot}\right)}{1-\cos^{2}\angle\left(v,V^{\bot}\right)}}\,\left\Vert u_{1}\right\Vert \left\Vert v_{1}\right\Vert }$
whenever $\left\Vert u\right\Vert \neq O$. The first part of the
theorem has been proved without the assumption $\dim V\geq2$.

The assumptions for the second part of the theorem include $v\neq O$ and $0<\angle\left(v,V^{\bot}\right)$. Therefore, $v_{1} \neq O$.
Another assumption of the second part of the theorem is $\dim V\geq2$, and so $z\in V$ can be chosen such that $\left\Vert z\right\Vert =1$
and $z\perp v_{1}$. It is straightforward to check that for $u=\cos\varphi_{1}\, v_{1}+\left\Vert v_{1}\right\Vert \sin\varphi_{1}\, z+\frac{1}{\cos\varphi_{1}}\, v_{2}$:
\begin{align}
\left\Vert u_{1}\right\Vert  & =\sqrt{\cos\varphi_{1}^{2}\left\Vert v_{1}\right\Vert ^{2}+\left\Vert v_{1}\right\Vert ^{2}\sin^{2}\varphi_{1}\left\Vert z\right\Vert ^{2}}=\left\Vert v_{1}\right\Vert\neq0 \label{eq:norm_special_u1}\\
\left\langle u_{1},v_{1}\right\rangle  & =\cos\varphi_{1}\,\left\Vert v_{1}\right\Vert ^{2}=\cos\varphi_{1}\,\left\Vert u_{1}\right\Vert \left\Vert v_{1}\right\Vert \nonumber \\
\left\Vert u\right\Vert  & =\sqrt{\left\Vert u_{1}\right\Vert ^{2}+\frac{\left\Vert v_{2}\right\Vert ^{2}}{\cos^{2}\varphi_{1}}}\overset{(\ref{eq:norm_special_u1})}{=}\left\Vert v_{1}\right\Vert \sqrt{1+\frac{\left\Vert v_{2}\right\Vert ^{2}}{\left\Vert v_{1}\right\Vert ^{2}\cos^{2}\varphi_{1}}}\nonumber \\
 & \overset{(\ref{eq:phi1})}{\underset{(\ref{eq:useful_identity})}{=}}\left\Vert v_{1}\right\Vert \frac{\cos\varphi}{\sqrt{\cos^{2}\varphi-\cos^{2}\angle\left(v,V^{\bot}\right)}}\label{eq:norm_special_u}\\
\left\Vert v\right\Vert  & =\left\Vert v_{1}\right\Vert \sqrt{1+\frac{\left\Vert v_{2}\right\Vert ^{2}}{\left\Vert v_{1}\right\Vert ^{2}}}\overset{(\ref{eq:useful_identity})}{=}\frac{\left\Vert v_{1}\right\Vert }{\sqrt{1-\cos^{2}\angle\left(v,V^{\bot}\right)}}\label{eq:norm_special_v}
\end{align}
Then, 
\begin{align*}
\left\langle u,v\right\rangle  & =\cos\varphi_{1}\,\left\Vert v_{1}\right\Vert ^{2}+\frac{\left\Vert v_{2}\right\Vert ^{2}}{\cos\varphi_{1}}=\left\Vert v_{1}\right\Vert ^{2}\left(\cos\varphi_{1}+\frac{\left\Vert v_{2}\right\Vert ^{2}}{\cos\varphi_{1}\,\left\Vert v_{1}\right\Vert ^{2}}\right)\\
 & \overset{(\ref{eq:phi1})}{\underset{(\ref{eq:useful_identity})}{=}}\left\Vert v_{1}\right\Vert ^{2}\frac{\cos^{2}\varphi}{\sqrt{\cos^{2}\varphi-\cos^{2}\angle\left(v,V^{\bot}\right)}\sqrt{1-\cos^{2}\angle\left(v,V^{\bot}\right)}}\overset{(\ref{eq:norm_special_u})}{\underset{(\ref{eq:norm_special_v})}{=}}\cos\varphi\left\Vert u\right\Vert \left\Vert v\right\Vert 
\end{align*}
Thus, when $\dim V\geq2$, formula (\ref{eq:phi1}) gives the smallest
possible $\varphi_{1}\in[0,\pi]$ and $\alpha_1=\cos\varphi_1$ is the largest possible in (\ref{eq:implikacija_projekcija}). 
\end{proof}
From Theorem \ref{thm:inequality} we were able to get a useful estimate
on the angle between projections when $\angle\left(u,v\right)<\angle\left(v,V^{\bot}\right)$.
On the other hand, the following result shows that when the angle
between vectors $\angle\left(u,v\right)>\angle\left(v,V^{\bot}\right)$,
then no useful estimate on the angle between projections can be provided.
In other words, the worst case scenario $\angle\left(Pu,Pv\right)=\pi$
is possible.
\begin{proposition}
\label{prop:phi_geq_psi}Let $V$ be a closed subspace of Hilbert
space $H$, with $\dim V\geq1$ and $V\neq H$. Let $P$ be the orthogonal
projection onto $V$. Let $v\in H$, $v\neq O$, $\varphi\in\left\langle \angle\left(v,V^{\bot}\right),\pi\right]$, and $\alpha=\cos\varphi$.
Then there exists $u\in H$ such that 
\begin{equation}
Pu\neq O, \quad\left\langle u,v\right\rangle \geq\alpha\,\left\Vert u\right\Vert \left\Vert v\right\Vert \quad\mbox{and}\quad\left\langle Pu,\, Pv\right\rangle =\left(-1\right)\,\left\Vert Pu\right\Vert \left\Vert Pv\right\Vert \mbox{.}\label{eq:conclusion_of_prop1-1}
\end{equation}
\end{proposition}
\begin{proof}
We use notation from Remark \ref{rem:decomposition}. 
Assumption of the proposition is that either $v_{1}\neq O$ or $v_{2}\neq O$. We will prove the proposition on a case by case basis.
We will choose a different continuous parametrization $u(t)$ for each case, and then investigate the continuous function
\begin{equation}
\label{function-phi-geq-psi}
f(t)\overset{{\rm def}}{=}\left\langle u(t),v\right\rangle -\cos\varphi\,\left\Vert u(t)\right\Vert \left\Vert v\right\Vert.
\end{equation}

Consider the case $v_{1}\neq O$ and $v_{2}\neq O$. Then $\nicefrac{\pi}{2}>\angle\left(v,V^{\bot}\right)>0$.
Let $u(t)=tv_{1}\dotplus v_{2}$ and $f$ as in \eqref{function-phi-geq-psi}. Then
\begin{multline}
f(t)=\left\Vert v_{1}\right\Vert ^{2}t+\left\Vert v_{2}\right\Vert ^{2}-\cos\varphi\,\sqrt{t^{2}\left\Vert v_{1}\right\Vert ^{2}+\left\Vert v_{2}\right\Vert ^{2}}\sqrt{\left\Vert v_{1}\right\Vert ^{2}+\left\Vert v_{2}\right\Vert ^{2}}\\
\overset{\eqref{eq:definition-Psi}}{=}\left\Vert v_{1}\right\Vert ^{2}t+\left\Vert v_{2}\right\Vert ^{2}
-\cos\varphi\,\sqrt{t^{2}\left\Vert v_{1}\right\Vert ^{2}+\left\Vert v_{2}\right\Vert ^{2}}\frac{\left\Vert v_{2}\right\Vert }{\cos\angle\left(v,V^{\bot}\right)}\mbox{.}\label{eq:f_definition-1}
\end{multline}
Plugging $t=0$ in (\ref{eq:f_definition-1})
gives 
$
f(0) =\left\Vert v_{2}\right\Vert ^{2}-{\displaystyle \dfrac{\cos\varphi}{\cos\angle\left(v,V^{\bot}\right)}}\left\Vert v_{2}\right\Vert ^{2}>0
$.
By the continuity of $f$, there exists some $t_{0}<0$ such that
$f(t_{0})>0$. Now $u=u(t_{0})$ satisfies (\ref{eq:conclusion_of_prop1-1}): the first part as $Pu=t_{0}v_1\neq O$,
the second part because $0<f(t_{0})=\left\langle u(t_{0}),v\right\rangle -\cos\varphi\,\left\Vert u(t_{0})\right\Vert \left\Vert v\right\Vert $
and the third as $\left\langle u_{1},v_{1}\right\rangle =\left\langle t_{0}v_{1},v_{1}\right\rangle =
-\left\Vert t_{0}v_{1}\right\Vert \left\Vert v_{1}\right\Vert =-\left\Vert u_{1}\right\Vert \left\Vert v_{1}\right\Vert $.

Next, consider the case $v_{1} = O$ and $v_{2}\neq O$. 
As $\dim V \geq 1$ we can choose $z\in V$ such that $\left\Vert z \right\Vert=1$.
Let $u(t)=t z\dotplus v_{2}$ and $f$ as in \eqref{function-phi-geq-psi}. Then
\begin{equation}
f(t)=\left\Vert v_{2}\right\Vert ^{2}-\cos\varphi\,\sqrt{t^{2}+\left\Vert v_{2}\right\Vert ^{2}}\,\left\Vert v_{2}\right\Vert 
\mbox{ and } f(0)=\left\Vert v_{2}\right\Vert ^{2}\left(1-\cos\varphi\right)\mbox{.}
\end{equation}
As $\varphi>\angle\left(v,V^{\bot}\right)=0$ thus $f(0)>0$. By the continuity of $f$, there exists some $t_{0}<0$ such that
$f(t_{0})>0$. Vector $u=u(t_{0})$ satisfies (\ref{eq:conclusion_of_prop1-1}) because $Pu=t_{0}z\neq O$,
 $0<f(t_{0})=\left\langle u,v\right\rangle -\cos\varphi\,\left\Vert u\right\Vert \left\Vert v\right\Vert $
and $\left\langle u_{1},v_{1}\right\rangle = 0 =-\left\Vert u_{1}\right\Vert \left\Vert v_{1}\right\Vert $.

Consider the final case $v_{1}\neq O$ and $v_{2}=O$. Then $\pi\geq\varphi>\angle\left(v,V^{\bot}\right)=\nicefrac{\pi}{2}$.
As $V\neq H$ there exists $z\in V^{\bot}$ such that 
$\left\Vert z\right\Vert =1$. Let $u(t)=tv_{1}\dotplus z$ and $f$ as in \eqref{function-phi-geq-psi}. Then
\[
f(t)=\left\Vert v_{1}\right\Vert ^{2}t-\cos\varphi\, \sqrt{t^{2}\left\Vert v_{1}\right\Vert ^{2}+1}\,\left\Vert v_{1}\right\Vert\mbox{.}
\]
As $\cos\varphi<0$ therefore $f(0)=-\cos\varphi\left\Vert v_{1}\right\Vert >0$.
By the continuity of $f$ there exists some $t_{0}<0$ such that $f(t_{0})>0$.
Now $u=u(t_{0})$ satisfies (\ref{eq:conclusion_of_prop1-1}):
the first part as $u_1=t_{0} v_1\neq O$, the second part because $0<f(t_{0})=\left\langle u,v\right\rangle -\cos\varphi\,\left\Vert u\right\Vert \left\Vert v\right\Vert $
and the third as $\left\langle u_{1},v_{1}\right\rangle =\left\langle t_{0}v_{1},v_{1}\right\rangle =-\left\Vert t_{0}v_{1}\right\Vert \left\Vert v_{1}\right\Vert =
-\left\Vert u_{1}\right\Vert \left\Vert v_{1}\right\Vert $. 
\end{proof}
The previous two results discussed the cases $\angle\left(u,v\right)<\angle\left(v,V^{\bot}\right)$
and $\angle\left(u,v\right)>\angle\left(v,V^{\bot}\right)$. The border
case $\angle\left(u,v\right)=\angle\left(v,V^{\bot}\right)$ is different,
as:
\begin{itemize}
\item $\inf\angle\left(Pu,Pv\right)=\nicefrac{\pi}{2}$ (combine Proposition
\ref{prop:phi_equal_psi-first_result} and Lemma \ref{lemm:phi_equal_psi-helper}), 
\item the infimum is achieved in the case $v\in V$ (see Proposition \ref{prop:phi_equal_psi-second_result}),
\item the infimum is not achieved in the case $v\notin V$ (see Proposition \ref{prop:phi_equal_psi-second_result}).\end{itemize}
\begin{proposition}
\label{prop:phi_equal_psi-first_result}Let $P$ be an orthogonal projection onto
a closed subspace $V$ of Hilbert space $H$, with $\dim V\geq2$.
Let $v\in H$, $v\neq O$, $\angle\left(v,V^{\bot}\right)>0$,
and $\varepsilon\in\left\langle 0,1\right]$. Then there exists $u\in H$
such that $P u\neq O$, 
\[
\left\langle u,v\right\rangle \geq\cos\angle\left(v,V^{\bot}\right)\,\left\Vert u\right\Vert \left\Vert v\right\Vert \quad\mbox{ and }\quad\left\langle Pu,Pv\right\rangle =\varepsilon\,\left\Vert Pu\right\Vert \left\Vert Pv\right\Vert \mbox{.}
\]
\end{proposition}
\begin{proof}
We use notation from Remark \ref{rem:decomposition}. As $\angle\left(v,V^{\bot}\right)>0$,
we have $\left\Vert v_{1}\right\Vert >0$ . Without the loss of generality
we assume that $\left\Vert v_{1}\right\Vert =1$ (if the statement
of the proposition is true for vector $\nicefrac{v}{\left\Vert v_{1}\right\Vert }$
then it is also true for $v$).

As $\dim V\geq2$, there exists $z\in V$ such that $z\perp v_{1}$
and $\left\Vert z\right\Vert =\sqrt{1-\varepsilon^{2}}$. Let $u(t)\overset{{\rm def}}{=}\varepsilon v_{1}+tv_{2}+z$ and
$u_1(t)\overset{{\rm def}}{=}P\ u(t)$.
Then $u_1(t)=\varepsilon v_{1}+z\neq O$, $\left\Vert u_1(t)\right\Vert =\sqrt{\varepsilon^{2}\left\Vert v_{1}\right\Vert ^{2}+\left\Vert z\right\Vert ^{2}}=1$,
and $\left\langle u_1(t),\, v_{1}\right\rangle =\varepsilon=\varepsilon\left\Vert u_1(t)\right\Vert \left\Vert v_{1}\right\Vert $.
So we only need to find $t$ such that $\left\langle u(t),v\right\rangle \geq\cos\angle\left(v,V^{\bot}\right)\,\left\Vert u(t)\right\Vert \left\Vert v\right\Vert $. 

We investigate the real function 
\[
f(t)\overset{{\rm def}}{=}\left\langle u(t),v\right\rangle -\cos\angle\left(v,V^{\bot}\right)\,\left\Vert u(t)\right\Vert \left\Vert v\right\Vert =\varepsilon+t\left\Vert v_{2}\right\Vert ^{2}-\left\Vert v_{2}\right\Vert \sqrt{1+t^{2}\left\Vert v_{2}\right\Vert ^{2}}\,.
\]
Note that $f$ is differentiable and strictly increasing, with ${\displaystyle \lim_{t\to+\infty}f(t)=\varepsilon>0}$.
Therefore, $f$ assumes positive value for some $t_{0}\in\mathbb{R}$.
Vector $u(t_{0})$ satisfies the conclusion of the proposition. \end{proof}
\begin{lemma}
\label{lemm:phi_equal_psi-helper}Let $P$ be an orthogonal projection
onto a closed subspace $V$ of Hilbert space $H$, with $V\neq H$.
Let $v\in H$, $v\neq O$ such that $\angle\left(v,V^{\bot}\right)>0$.
Then 
\[
\forall u\in H,\qquad\left\langle u,v\right\rangle \geq\cos\angle\left(v,V^{\bot}\right)\,\left\Vert u\right\Vert \left\Vert v\right\Vert \quad\Longrightarrow\quad\left\langle Pu,Pv\right\rangle \geq0\mbox{.}
\]
\end{lemma}
\begin{proof}
If $V=\left\{ O\right\} $, $v=O$, $Pv=O$ ($\angle(v,V^{\bot})=0$)
or $Pu=O$ the conclusion of the lemma is trivial. Otherwise, apply
Lemma \ref{lemm:phi_equal_psi-helper-2}.\end{proof}
\begin{lemma}
\label{lemm:phi_equal_psi-helper-2}Let $P$ be an orthogonal projection
onto a closed subspace $V$ of Hilbert space $H$, with $V\neq H$
and $V\neq\left\{ O\right\} $. Let $v\in H$, such that $0<\angle\left(v,V^{\bot}\right)<\nicefrac{\pi}{2}$.
Then for all $u\in H$,
\[
\left(\left\langle u,v\right\rangle \geq\cos\angle\left(v,V^{\bot}\right)\,\left\Vert u\right\Vert \left\Vert v\right\Vert \mbox{ and }Pu\neq O\right)\quad\Longrightarrow\quad\left\langle Pu,Pv\right\rangle >0\mbox{.}
\]
\end{lemma}
\begin{proof}
We use notation from Remark \ref{rem:decomposition}. As $0<\angle\left(v,V^{\bot}\right)<\nicefrac{\pi}{2}$
so $v\neq O$, $v_{1}\neq O$, and $v_{2}\neq O$.

To prove the conclusion of the lemma, we assume $u_{1}\neq O$ and
\begin{equation}
\left\langle u_{1},v_{1}\right\rangle +\left\langle u_{2},v_{2}\right\rangle =\left\langle u,v\right\rangle \geq\cos\angle\left(v,V^{\bot}\right)\,\left\Vert u\right\Vert \left\Vert v\right\Vert \overset{\eqref{eq:definition-Psi}}{=}\left\Vert u\right\Vert \left\Vert v_{2}\right\Vert \label{eq:pomocna1-1}
\end{equation}
From (\ref{eq:pomocna1-1}), as $\left\Vert u_{2}\right\Vert <\left\Vert u\right\Vert $
we get 
\[
\left\Vert u\right\Vert \left\Vert v_{2}\right\Vert -\left\langle u_{1},v_{1}\right\rangle \overset{(\ref{eq:pomocna1-1})}{\leq}\left\langle u_{2},v_{2}\right\rangle \leq\left\Vert u_{2}\right\Vert \left\Vert v_{2}\right\Vert <\left\Vert u\right\Vert \left\Vert v_{2}\right\Vert .
\]
Thus, $\left\langle u_{1},v_{1}\right\rangle >0$ follows by subtraction
of $\left\Vert u\right\Vert \left\Vert v_{2}\right\Vert $ from both
sides of the previous inequality. \end{proof}
\begin{proposition}
\label{prop:phi_equal_psi-second_result}Let $P$ be an orthogonal
projection onto a closed subspace $V$ of Hilbert space $H$, such
that $V\neq H$ and $\dim V\geq2$. Let $v\in H$, such that $v\neq O$
and $\angle(v,V^{\bot})>0$. Then $\angle(v,V^{\bot})<\nicefrac{\pi}{2}$
if and only if for all $u\in H$, 
\begin{equation}
\left(\left\langle u,v\right\rangle \geq\cos\angle(v,V^{\bot})\,\left\Vert u\right\Vert \left\Vert v\right\Vert \mbox{ and }Pu\neq O\right)\enskip\Longrightarrow\enskip\left\langle Pu,Pv\right\rangle >0\mbox{.}\label{eq:second_result}
\end{equation}
\end{proposition}
\begin{proof}
One implication is proved in lemma \ref{lemm:phi_equal_psi-helper-2}.

Instead of proving that (\ref{eq:second_result}) implies $\angle(v,V^{\bot})<\nicefrac{\pi}{2}$,
we prove the contraposition: $\angle(v,V^{\bot})\geq\nicefrac{\pi}{2}$
implies that there exists an $u\neq O$ such that $\left\langle u,v\right\rangle \geq\cos\angle\left(v,V^{\bot}\right)\,\left\Vert u\right\Vert \left\Vert v\right\Vert $,
$Pu\neq O$ and $\left\langle Pu,\, Pv\right\rangle \leq0$. From (\ref{eq:definition-Psi}),
$\angle\left(v,V^{\bot}\right)\geq\nicefrac{\pi}{2}$, the proposition assumptions $V\neq H$ and
$v\neq O$, we can conclude that $\angle\left(v,V^{\bot}\right)=\nicefrac{\pi}{2}$.
Therefore $\cos\angle\left(v,V^{\bot}\right)=0$ and $v\in V$. As
$\dim V\geq2$ there exists $z\in V$ such that $z\perp v$ and $\left\Vert z\right\Vert =1$.
For $u=z$ we get $\left\langle u,v\right\rangle =0\geq\cos\angle(v,V^{\bot})\,\left\Vert u\right\Vert \left\Vert v\right\Vert $
and $\left\langle Pu,\, Pv\right\rangle =\left\langle z,v\right\rangle =0$.
\end{proof}
The following lemma allows us to go from just one $u\in C_{H}\left(v,\varphi\right)$
such that $Pu$ is on the boundary of $C_{V}\left(P\, v,\theta\right)$
and conclude that all of $C_{V}\left(P\, v,\theta\right)$ is inside
of the projection of $C_{H}\left(v,\varphi\right)$.
\begin{lemma}
\label{lem:extreme-bound-to-projected-cone} Let $\theta\in\left[0,\pi\right]$,
$u\in C_{H}\left(v,\varphi\right)$, $P\, u\neq O$
and $\left\langle P\, u,P\, v\right\rangle =\cos\theta\left\Vert P\, u\right\Vert \left\Vert P\, v\right\Vert $.
Then $C_{V}\left(P\, v,\theta\right)\subseteq P\left[C_{H}\left(v,\varphi\right)\right]$. \end{lemma}
\begin{proof}
We use notation from Remark \ref{rem:decomposition}. Note that $u_{1}\neq O$, and
so if $\varphi=0$, then the case $v_{1}= 0$ is excluded from the conclusion of the lemma. On the 
other hand, when $\varphi\neq 0$ then the case $v_{1}= 0$ is included in the lemma. 
Notice that $u\in C_{H}\left(v,\varphi\right)$ corresponds
to $\left\langle u,v\right\rangle \geq\cos\varphi\left\Vert u\right\Vert \left\Vert v\right\Vert $.

We will prove that for each $w\in V$ such that 
$\left\langle w,v_{1}\right\rangle \geq\cos\theta\left\Vert w\right\Vert \left\Vert v_{1}\right\Vert $,
there exists $z\in V^{\perp}$ such that $\tilde{u}\overset{{\rm def}}{=}w\dotplus z$,
that satisfies $\left\langle \tilde{u},v\right\rangle \geq\cos\varphi\left\Vert \tilde{u}\right\Vert \left\Vert v\right\Vert $.
Note that by the definition of $\tilde{u}$ we get $P\tilde{u}=w$.
Take ${\displaystyle z=u_{2}\frac{\left\Vert w\right\Vert }{\left\Vert u_{1}\right\Vert }}$,
then ${\displaystyle \left\Vert \tilde{u}\right\Vert =\frac{\left\Vert w\right\Vert }{\left\Vert u_{1}\right\Vert }\,\left\Vert u\right\Vert }$
and 
\begin{multline*}
\left\langle \tilde{u},v\right\rangle =\left\langle w,v_{1}\right\rangle +\frac{\left\Vert w\right\Vert }{\left\Vert u_{1}\right\Vert }\left\langle u_{2},v_{2}\right\rangle \geq\cos\theta\left\Vert w\right\Vert \left\Vert v_{1}\right\Vert +\frac{\left\Vert w\right\Vert }{\left\Vert u_{1}\right\Vert }\left\langle u_{2},v_{2}\right\rangle =\\
=\frac{\left\Vert w\right\Vert }{\left\Vert u_{1}\right\Vert }\left(\cos\theta\left\Vert u_{1}\right\Vert \left\Vert v_{1}\right\Vert +\left\langle u_{2},v_{2}\right\rangle \right)=\frac{\left\Vert w\right\Vert }{\left\Vert u_{1}\right\Vert }\left(\left\langle u_{1},v_{1}\right\rangle +\left\langle u_{2},v_{2}\right\rangle \right)=\\
=\frac{\left\Vert w\right\Vert }{\left\Vert u_{1}\right\Vert }\left\langle u,v\right\rangle \geq\frac{\left\Vert w\right\Vert }{\left\Vert u_{1}\right\Vert }\cos\varphi\,\left\Vert u\right\Vert \left\Vert v\right\Vert =\cos\varphi\,\left\Vert \tilde{u}\right\Vert \left\Vert v\right\Vert \mbox{.}
\end{multline*}\end{proof}

\Acknowledgements{This paper is a result of author's work on doctoral dissertation
at University of Zagreb. The author would like to thank to his adviser,
prof. Josip Tamba\v{c}a. The author would also like thank to prof. Sever
Dragomir and prof. Marko Mati\'{c}, for their helpful comments.}

\bigskip

\bigskip

\begin{center}
{\bf Ortogonalna projekcija beskona\v{c}nog kru\v{z}nog konusa u realnom Hibertovom prostoru}
\end{center}

\bigskip

\begin{center}
{\it Mate Kosor}
\end{center}

\bigskip

\begin{center}
\begin{minipage}[c]{9.2cm}
{\small \hspace*{0.5cm} {\sc Sa\v{z}etak.}
Dajemo potpuni opis ortogonalnih projekcija beskona\v{c}nog kru\v{z}nog sto\v{s}ca u realnim 
Hilbertovim prostorima. Druga interpretacija je da smo za dva vektora dobili optimalnu 
ocjenu kuta izme\dj{}u ortogonalnih projekcija tih vektora. Ta ocjena ovisi o kutu izme\dj{}u 
polazna dva vektora i polo\v{z}aju samo jednog od njih. Me\dj{}u rezultatima je tako\dj{}er
doprinos nejednakostima tipa Cauchy-Bunyakovsky-Schwarz.%
}
\end{minipage}
\end{center}

\endarticle 

\end{document}